\documentclass{amsart}
\usepackage[utf8]{inputenc}

\date{\today}

\title{Maximal sets without choice}
\author{Jonathan Schilhan}
\email{j.schilhan@leeds.ac.uk}
\address{School of Mathematics, University of Leeds, Leeds, LS2 9JT, UK}
\thanks{The author was supported by a UKRI Future Leaders Fellowship MR/T021705/1}
\usepackage{amsmath}
\usepackage{amssymb}
\usepackage{mathrsfs}
\usepackage{tikz-cd}
\usepackage{enumitem}
\usepackage{url}
\usepackage{hyperref}
\usepackage{bbm}
\usepackage{ stmaryrd }
\usepackage{todonotes}

    \DeclareMathOperator{\tp}{tp}

    \DeclareMathOperator{\dom}{dom}

    \DeclareMathOperator{\Coll}{Coll}
    
    \DeclareMathOperator{\cf}{cf}
    \DeclareMathOperator{\sym}{sym}
\newcommand{\axiom}[1]{\mathsf{#1}}
\newcommand{\ZFC}{\axiom{ZFC}}

\newcommand{\AC}{\axiom{AC}}

\newcommand{\DC}{\axiom{DC}}
\newcommand{\ZF}{\axiom{ZF}}

\newcommand{\HOD}{\mathrm{HOD}}
\newcommand{\GCH}{\axiom{GCH}}

\newcommand{\HS}{\axiom{HS}}

\newcommand{\WO}{\axiom{WO}}
\newcommand{\CUP}{\axiom{CU}(\mathbf\Sigma^1_\omega)}
\newcommand{\CU}{\axiom{CU}}
\newcommand{\SP}{\axiom{SP}}
\newcommand{\HYP}{\axiom{HYP}}

\newcommand{\sF}{\mathscr F}
\newcommand{\sG}{\mathscr G}

\theoremstyle{plain}
\newtheorem{thm}{Theorem}[section]
\newtheorem{lemma}[thm]{Lemma}
\newtheorem{prop}[thm]{Proposition}
\newtheorem{cor}[thm]{Corollary}

\theoremstyle{definition}
\newtheorem{definition}[thm]{Definition}

\newtheorem{remark}[thm]{Remark}
\newtheorem{quest}[thm]{Question}
 
\begin{document}

\maketitle

\begin{abstract}
We show that it is consistent relative to $\ZF$, that there is no well-ordering of $\mathbb{R}$ while a wide class of special sets of reals such as Hamel bases, transcendence bases, Vitali sets or Bernstein sets exists. To be more precise, we can assume that every projective hypergraph on $\mathbb{R}$ has a maximal independent set, among a few other things. For example, we get transversals for all projective equivalence relations. Moreover, this is possible while either $\DC_{\omega_1}$ holds, or countable choice for reals fails. Assuming the consistency of an inaccessible cardinal, ``projective" can even be replaced with ``$L(\mathbb{R})$". 
This vastly strengthens the consistency results obtained in \cite{brendle2018model}, \cite{Horowitz2020} or \cite{KanoveiSchindler}.
\end{abstract}


\section{Introduction}

Using the \emph{Axiom of Choice}, or more specifically, a \emph{well-ordering} of the continuum special sets of reals such as Hamel bases, transcendence bases, Vitali sets or two-point sets (also known as Mazurkiewicz sets) can be constructed. In many cases, it was unknown for a long time whether the existence of a well-ordering of $\mathbb{R}$ is necessary to have such sets. For example, it has only recently been shown in \cite{Beriashvili2018} that the existence of a Hamel basis is indeed consistent with the non-existence of a well-ordering. This answers a question of Pincus and P\v rikrý from the 70's (see \cite[p. 433]{Pincus1975}). In \cite{Horowitz2020}, the same is shown for transcendence bases, answering a question of Larson and Zapletal. In \cite{brendle2018model}, a model is constructed in which a Hamel basis, a Vitali set, a Bernstein, Luzin and Sierpi\' nski set exist all simultaneously while the continuum can't be well-ordered. Here, recall that a \emph{Bernstein set} is a set $X$ of reals such that neither $X$ nor its complement contain an uncountable closed subset. A \emph{Luzin set} is an uncountable set of reals that intersects every meager set in only a countable set and a \emph{Sierpi\' nski} set is uncountable and intersects every measure zero set in only countably many points. In \cite{KanoveiSchindler} moreover, the authors show that a $\Delta^1_3$-definable Hamel basis can exist while countable choice for reals fails.
Similar types of results are also obtained in \cite{Larson2020} or \cite{Larson2017} and it is fair to say, that interest in questions of this type has greatly increased in recent years. 

The goal of the following paper, is to present models that supersede the previously mentioned results and improve them to a great extent. Many of the examples that we mention above can be treated as \emph{maximal independent} sets in \emph{hypergraphs} on $\mathbb{R}$ or other suitable Polish spaces. A hypergraph on a set $X$ is a set $E$ of finite non-empty subsets of $X$. Then, $Y \subseteq X$ is called $E$-independent if no ``edge" in $E$ can be formed by elements of $Y$, i.e. $[Y]^{<\omega} \cap E = \emptyset$. Moreover, $Y$ is maximal $E$-independent if it can not be properly extended to an $E$-independent set. For example, when $X= \mathbb{R}$ and $E$ consists of finite linearly dependent subsets of $\mathbb{R}$ over $\mathbb{Q}$, a maximal independent set is a Hamel basis. When $E$ is an equivalence relation, a maximal $E$-independent set is also called a \emph{transversal} for $E$. So when $E$ consists of $\{x,y\} \subseteq \mathbb{R}$ such that $x \neq y$ and $x-y \in \mathbb{Q}$, we get Vitali sets. Similarly, we can frame many other examples such as ultrafilters or transcendence bases. 

One of the main results of \cite{Schilhan2020} is the following.  

\begin{thm}\label{thm:mainold}
Assume $V = L$. Let $\mathbb{P}$ be a countable support iteration or a finite product of Sacks forcing. Then, in $V^\mathbb{P}$, every analytic hypergraph has a $\mathbf\Delta^1_2$ definable maximal independent set. 
\end{thm}

If $\mathbb{P}$ is the countable support iteration of Sacks forcing of length $\omega_1$, then it is well-known that in $L(\mathbb{R})$ of the forcing extension $V^\mathbb{P}$, there is no well-ordering of the reals and the \emph{Axiom of Dependent Choice}, which we denote by $\DC$, holds.\footnote{For a definition of $\DC$ see Definition~\ref{def:DC}.} Moreover, Luzin sets and Sierpi\' nski sets that exist in $L$ remain such sets by well-known preservation results such as can be found in \cite{BartoszynskiJudah1995}. Let us write $\WO(\mathbb{R})$ to say that the reals can be well-ordered. Then we immediately get the following corollary of Theorem~\ref{thm:mainold}.

\begin{cor}\label{cor:oldmain}
The following is consistent relative to $\ZF$: 
\begin{enumerate}
    \item $\DC + \neg \WO(\mathbb{R})$.
    \item Every analytic hypergraph on a Polish space has a ($\mathbf\Delta^1_2$ definable) maximal independent set.  
    \item There is a Luzin and a Sierpi\' nski set.
\end{enumerate}
\end{cor}

In particular, we consistently also get ultrafilters, Hamel bases, transcendence bases, Vitali sets and mad families, just to name a few examples, all simultaneously without a well-ordering of $\mathbb{R}$. 
This already improves most of the results we mentioned previously.

We will extend Corollary~\ref{cor:oldmain} in a few directions. First, we may expand the class of hypergraphs to include all projective sets. Crucially, this will use a result by Shelah that a certain projective uniformization priciple relative to comeager sets is consistent relative to $\ZFC$. Moreover, assuming the consistency of an inaccessible cardinal we can consider all hypergraphs in $L(\mathbb{R})$. Secondly, we will show that $\AC_\omega(\mathbb{R})$, countable choice for reals (see Definition~\ref{def:ACX}), may fail. $\WO(\mathbb{R})$ can thus fail in a particularily strong way. On the other hand, we can also obtain a model in which $\DC_{\omega_1}$ holds. Our main result is the following: 

\begin{thm}\label{thm:main1}
The conjunction of the following is consistent with $\DC_{\omega_1}$ as well as with $\neg \AC_{\omega}(\mathbb{R})$ relative to $\ZF$: 
\begin{enumerate}
    \item There is no well-ordering of the continuum. 
    \item Every projective hypergraph on the reals has a maximal independent set that is a union of $\aleph_1$-many compact sets.
    \item In particular, there is a Hamel basis, a transcendence basis, a Vitali set, etc.
    \item Every projective set has the Baire property. 
    \item There is a two-point set, a tower, a scale, a P-point, a Bernstein set, a Luzin set and a Sierpiński set. 
\end{enumerate}

Assuming an inaccesible cardinal, we can replace ``projective" by ``$L(\mathbb{R})$" and add: 

\begin{enumerate}
\setcounter{enumi}{5}
    \item Every equivalence relation in $L(\mathbb{R})$ has a transversal. 
\end{enumerate}
\end{thm}

In yet another direction, we seek to improve the results of Kanovei and Schindler from \cite{KanoveiSchindler} by lowering the complexity of a $\Delta^1_3$ Hamel basis to $\Delta^1_2$. Unfortunately, while in the model of \cite{KanoveiSchindler} there is a $\Delta^1_3$ Bernstein set, we only find a $\Delta^1_4$ such set. 

\begin{thm}\label{thm:main2}
The conjunction of the following is consistent with $\ZF + \neg\AC_\omega(\mathbb{R})$ as well as with $\ZF + \DC + \neg\WO(\mathbb{R})$.
\begin{enumerate}
    \item Every $\mathbf\Sigma^1_1$ hypergraph has a $\mathbf\Delta^1_2$ maximal independent set.
    \item In particular, there is a $\Delta^1_2$ Hamel basis, transcendence basis, Vitali set, etc. 
    \item There is a $\Pi^1_1$ tower and scale.
    \item There is a $\Delta^1_2$ Luzin set, Sierpiński set and P-point.
    \item There is a $\Delta^1_4$ Bernstein set.
    \end{enumerate}
\end{thm} 

The models are obtained using the method of symmetric extensions applied to iterated Sacks forcing (see Section~\ref{sec:symit}). Recall that a \emph{tower} is a $\subseteq^*$ decreasing sequence without a lower bound. A \emph{scale} is a $<^*$ increasing sequence that is dominating, i.e. for every $g \in \omega^\omega$ it contains $f$ such that $g <^* f$. As usual, for $x, y \in [\omega]^\omega$, $f,g \in \omega^\omega$, $x \subseteq^* y$ iff $x \setminus y$ is finite and $g <^* f$ iff $f(n) > g(n)$, for all but finitely many $n \in \omega$. A P-point is an ultrafilter $\mathcal{U}$ on $\omega$ such that for any countable $A \subseteq \mathcal{U}$, there is $x \in \mathcal{U}$ with $x \subseteq^* y$, for every $y \in A$.

\section{Preliminaries}

\subsection{Weak choice principles}

\begin{definition}\label{def:DC}
$\DC$ is the statement that every tree without maximal nodes has an infinite branch. 

More generally, for an ordinal $\lambda$, $\DC_\lambda$ is the statement that every $\lambda$-closed tree $T$ has a chain of length $\lambda$. We say that a tree $T$ is $\lambda$-closed if every chain of length $< \lambda$ in $T$ has a least upper bound.

We write $\DC_{<\lambda}$ for the statement that $\DC_\gamma$ holds for every $\gamma < \lambda$. 
\end{definition}

\begin{definition}\label{def:ACX}
$\AC_X(Y)$, for sets $X$ and $Y$, is the statement that every family $\langle A_x : x \in X \rangle$ of non-empty subsets of $Y$ has a choice function, i.e. a function $f$ with domain $X$, such that $f(x) \in A_x,$ for every $x \in X$.
\end{definition}

\subsection{Iterations of Sacks forcing}

\begin{definition}[Sacks forcing]
Sacks forcing $\mathbb{S}$ consists of perfect subtrees of $2^{<\omega}$ ordered by inclusion. 
\end{definition}

Sacks forcing as well as its finite powers have the continuous reading of names, a useful property that lets us examine forcing properties topologically. It provides a link between the descriptive set theory of perfect sets and the structure of the reals in the forcing extension. 

\begin{lemma}[Folklore, also see \cite{Geschke2004}]\label{lem:crnfinite} Let $k \in \omega$. Then $\mathbb{S}^k$ is proper and for every $\mathbb{S}^k$-name $\dot y$ for a real and any $\bar p \in \mathbb{S}^k$, there is $\bar q \leq \bar p$ and a continuous function $f$ from $[\bar q] := \prod_{l<k} [q(l)]$ to $\omega^\omega$, such that $$\bar q \Vdash \dot y = f(\bar x_{\operatorname{gen}}),$$ where $\bar x_{\operatorname{gen}}$ is a name for the generic element of $(2^\omega)^k$ added by $\mathbb{S}^k$.
\end{lemma}

If $\lambda$ is an ordinal, we denote with $\mathbb{S}^{*\lambda}$ the countable support iteration of $\mathbb{S}$ of length $\lambda$. In the case of $\mathbb{S}^{*\lambda}$ an appropriate version of the continuous reading of names is less clear since there is no immediate topological representation for its conditions. There is a way to consider conditions topologically and prove a version of continuous reading names though. Various approaches to this exist in the literature (e.g. in \cite{Geschke2004}, \cite{Kanovei1976} or \cite{Schrittesser2016})~and we present here a condensed version that we think is most accessible and suitable for our purposes. It is basically a reformulation of \cite[Lemma~2.3--2.5]{Schilhan2020}.\footnote{Particularly similar to our approach is the theory of ``normal sets" as presented in \cite{Budinas1983} or \cite{Kanovei1976}.}

Consider the forcing $\tilde{\mathbb{S}}^{*\lambda}$ consisting of conditions $(X,C)$, where $C \in [\lambda]^{\aleph_0}$ and $X$ is a non-empty closed subset of $(2^\omega)^C$ such that for every $\alpha \in C$ and $\bar x \in X \restriction \alpha := \{ \bar z \restriction \alpha : \bar z \in X \}$, $$T_{\bar x} := \{s \in 2^{<\omega} : \exists y\exists \bar z (s \subseteq y \wedge \bar x^\frown y^\frown \bar z \in X)\} \in \mathbb{S}$$ and the function from $X \restriction \alpha$ to $\mathcal{P}(2^{<\omega})$ mapping $\bar x$ to $T_{\bar x}$ is continuous. The order on $\tilde{\mathbb{S}}^{*\lambda}$ is $(Y,D) \leq (X,C)$ iff $C \subseteq D$ and $Y \restriction C \subseteq X$. Let $i_\lambda \colon \tilde{\mathbb{S}}^{*\lambda} \to \mathbb{S}^{*\lambda}$ be such that $i_\lambda(X,C) = \bar q$, where for each $\alpha < \lambda$, $q(\alpha)$ is a $\mathbb{S}^{*\alpha}$-name for $T_{\bar x_{\operatorname{gen}} \restriction (C \cap \alpha)}$ if $\alpha \in C$, where $\bar x_{\operatorname{gen}}$ is a name for the generic sequence of reals added by $\mathbb{S}^{*\lambda}$, or $q(\alpha)$ is a $\mathbb{S}^{*\alpha}$-name for $2^{<\omega} = \mathbbm{1}_{\mathbb{S}}$ if $\alpha \notin C$. Then we have the following: 

\begin{lemma}\label{lem:crninfinite}
For any $\lambda$, $i_\lambda$ is a dense embedding. Thus $\mathbb{S}^{*\lambda}$ is forcing equivalent to $\tilde{\mathbb{S}}^{*\lambda}$. Moreover, whenever $\dot y$ is a ${\mathbb{S}}^{*\lambda}$ name for an element of $\omega^\omega$ and $(X,C) \in \tilde{\mathbb{S}}^{*\lambda}$, there is $(Y,D) \leq (X,C)$ and a continuous function $f \colon Y \to \omega^\omega$ such that $$i_\lambda(Y,D) \Vdash \dot y = f(\bar x_{\operatorname{gen}} \restriction D),$$ where $\bar x_{\operatorname{gen}}$ is a name for the generic element of $(2^\omega)^\lambda$ added by $\mathbb{S}^{*\lambda}$. Further, if $A \subseteq \omega^\omega$ is any analytic set, we can ensure that $i_\lambda(Y,D) \Vdash \dot y \in A$ iff $f''Y \subseteq A$.
\end{lemma}

For a condition $\bar p \in \mathbb{S}^k$, $[\bar p]$ is easily seen to be homeomorphic to $(2^\omega)^k$ and for topological arguments it often makes sense to assume it is just equal to $(2^\omega)^k$. This will be useful later in stating Lemma~\ref{lem:gettingmcg}. Thus let us fix the following canonical homeomorphisms.

\begin{definition}\label{def:isom}
For a perfect tree $T \subseteq 2^{<\omega}$ we let $\eta_T \colon [T] \to 2^\omega$ be a canonical homeomorphism between $[T]$ and $2^\omega$, in a way that both $\eta_T(x)$ and $\eta_T^{-1}(x)$ depend continuously on $T$ and $x$. For $(X,C) \in \tilde{\mathbb{S}}^{*\lambda}$, let $\alpha := \tp(C)$ and $\iota \colon \alpha \to C$ be the unique order isomorphism. Then we let $\Phi_X \colon X \to (2^\omega)^\alpha$ be the map defined by $$\Phi_X(\bar x)(\beta) = \eta_{T_{\bar x \restriction \iota(\beta)}}(x(\iota(\beta))),$$ for every $\beta <\alpha$. It is not hard to see that this defines a homeomorphism.

Similarly, we define $\Phi_{\bar p} \colon [\bar p] \to (2^\omega)^k$ for $\bar p \in \mathbb{S}^k$, $k \in \omega$. 
\end{definition}

\subsection{Symmetric extensions}

To produce models of $\ZF$ in which $\AC$ fails, we use the method of symmetric extensions. This technique produces an intermediate model between $V$ and a generic extension $V[G]$ by only evaluating names from a specific carefully defined class of names. Let us review its basics. A more detailed exposition can be found in \cite{Jech1973}.

Let $\mathbb{P}$ be a forcing poset and $\pi$ an automorphism of $\mathbb{P}$. Then $\pi$ extends naturally to $\mathbb{P}$-names by recursion on the rank of a name $\dot x$: $$\pi(\dot x) = \{(\pi(p), \pi(\dot y)) : (p,\dot y) \in \dot x \}.$$

The following lemma is well-known (see e.g. \cite[Lemma~14.37]{Jech2003}).
\begin{lemma}[Symmetry Lemma]
Let $p\in \mathbb{P}$, $\dot x$ a $\mathbb{P}$-name and $\pi$ an automorphism of $\mathbb{P}$. Then $$p\Vdash\varphi(\dot x)\leftrightarrow\pi(p)\Vdash\varphi(\pi(\dot x)).$$
\end{lemma}

Let $\sG$ be a group of automorphisms of $\mathbb{P}$. For any $\mathbb{P}$-name $\dot x$ we define the subgroup $\sym_\sG(\dot x)$ of $\sG$ of automorphisms that fix $\dot x$. To be more precise, $$\sym_\sG(\dot x) = \{\pi \in \sG : \pi(\dot x) = \dot x \}.$$ A collection $\sF$ is called a \textit{filter of subgroups} of $\sG$ if it is a non-empty collection of subgroups of $\sG$ which is closed under supergroups and finite intersections. We say that $\sF$ is \textit{normal} if it is further closed under conjugation, i.e. for any $H\in\sF$ and $\pi\in\sG$, also $\pi H\pi^{-1}\in\sF$.

A \emph{symmetric system} is a triple $(\mathbb{P}, \sG, \sF)$ where $\mathbb{P}$ is a forcing notion, $\sG$ is a group of automorphisms of $\mathbb{P}$ and $\sF$ is a normal filter of subgroups of $\sG$. Given such a system, we say that a name $\dot x$ is \emph{$\sF$-symmetric} if $\sym_\sG(\dot x) \in \sF$. Moreover, by recursion on the rank of $\dot x$ we say that $\dot x$ is \emph{hereditarily $\sF$-symmetric} if for every $(p,\dot y) \in \dot x$, $\dot y$ is hereditarily $\sF$-symmetric. The class of hereditarily $\sF$-symmetric names will be denoted by $\HS_\sF$. Finally, given a $\mathbb{P}$-generic $G$ over $V$, we define the class $M := \{\dot x^G : \dot x \in \HS_\sF \}$. Then $M$ is a transitive model of $\ZF$ and $V \subseteq M \subseteq V[G]$ (see \cite[Lemma 15.51]{Jech2003}). $M$ is a so called \emph{symmetric extension} of $V$ produced by the symmetric system $(\mathbb{P}, \sG, \sF)$. 
Typically, in applications we will omit all the subscripts when $\sG$ and $\sF$ are clear from context.

We will now define the symmetric systems that we will use in the main results. 

\subsubsection{The symmetric Sacks iteration}
\label{sec:symit}
Let $\mathbb{P} = \mathbb{S}^{*\lambda}$ for some ordinal $\lambda$. We call an automorphism $\pi$ of $\mathbb{S}^{*\lambda}$ \emph{coherent} if $$\pi(\bar p) \restriction \alpha = \pi(\bar p \restriction \alpha)$$ for any $\bar p \in \mathbb{S}^{*\lambda}$ and $\alpha \leq \lambda$. Whenever $\bar \pi = \langle \dot \pi_\alpha : \alpha < \lambda \rangle$ is a sequence such that for every $\alpha < \lambda$, $\dot \pi_\alpha$ is an $\mathbb{S}^{*\alpha}$-name for an automorphism of $\mathbb{S}$, we use $\bar \pi$ to denote a coherent automorphism of $\mathbb{S}^{*\lambda}$ defined by induction on $\lambda$ as follows.\footnote{Something similar appears in \cite{Karagila2019} where it is denoted $\int_{\bar \pi}$.} If $\lambda = \alpha + 1$, we let $$\bar \pi(\bar p) = (\bar \pi \restriction \alpha)( \bar p \restriction \beta)^\frown \dot r,$$ where $\dot r$ is a name for $\dot \pi_\alpha^G\left((\bar \pi \restriction \alpha)(\dot p(\beta))^G\right)$, for a $\mathbb{S}^{*\alpha}$-generic $G$.\footnote{There are of course many equivalent names that one can pick for $\dot r$, but it does not matter which one we choose as long as we ensure that $\bar \pi$ is bijective.} If $\lambda$ is a limit ordinal, $$\bar \pi(\bar p) = \bigcup_{\alpha < \lambda} (\bar\pi \restriction\alpha)(\bar p \restriction\alpha).$$
It is easy to check by induction that this indeed forms a coherent automorphism of $\mathbb{S}^{*\lambda}$. In fact it is not hard to see that every coherent automorphism is of this form, although this won't be relevant to us. Let $\sG$ be the group of coherent automorphisms of $\mathbb{S}^{*\lambda}$ and $\sF$ be the filter generated by subgroups $$H_\alpha := \{ \pi \in \sG : \forall \bar p \in \mathbb{S}^{*\alpha} ( \pi(\bar p) = \bar p) \}$$ for $\alpha < \lambda$. 
To see that $\sF$ is normal it suffices to note that $\pi H_\alpha \pi^{-1} = H_\alpha$ for any $\alpha$ and coherent $\pi$.


\begin{lemma}$(\GCH)$
\label{lem:symit}
Let $\lambda$ be a limit ordinal of cofinality $\leq \omega_1$ or a regular cardinal. Let $\cf(\lambda) = \kappa$ and $M$ be a symmetric extension of $V$ produced by the symmetric system $(\mathbb{S}^{*\lambda}, \sG, \sF)$ and an $\mathbb{S}^{*\lambda}$-generic filter $G$. Finally, let $\langle x_\alpha : \alpha <\lambda \rangle$ be the generic sequence of reals added by $G$. Then:

\begin{enumerate}
    \item $M$ is closed under $<\kappa$-sequences in $V[G]$.
    \item $\DC_{< \kappa}$ holds in $M$. 
    \item For any set $X \in M$ of ordinals, $X \in V[\langle x_\beta : \beta < \alpha \rangle]$ for some $\alpha < \lambda$.
    \item There is no well-ordering of the reals in $M$.
    \item Letting $R_\alpha := \omega^\omega \cap V[\langle x_\beta : \beta < \alpha \rangle]$, $\langle R_\alpha : \alpha < \lambda \rangle \in M$.
    \item $\AC_{\lambda}(\mathbb{R})$ fails in $M$.
\end{enumerate}
\end{lemma}

\begin{proof}
(1): When $\kappa = \omega$, then the statement is trivial. Now assume that $\kappa = \omega_1$. Let $\langle \dot y_i^G : i < \xi \rangle \in V[G]$, $\xi < \omega_1$, where $\{\dot y_i^G : i < \xi \} \subseteq M$. In $V$, we find for each $i < \xi$ a maximal antichain $A_i \subseteq \mathbb{S}^{*\lambda}$ such that for each $\bar q \in A_i$, there is $\dot y_{i,\bar q} \in \HS$ such that $$\bar q \Vdash \dot y_i = \dot y_{i,\bar q}.$$ 
By properness and since $\kappa$ is regular uncountable, we find a condition $\bar p \in G$ and $\alpha < \lambda$ such that for every $i < \xi$, $ A_i \restriction \bar p := \{ \bar q \in A_i : \bar p \parallel \bar q \} \subseteq \mathbb{S}^{*\alpha}$ is countable and $H_\alpha \subseteq \sym(\dot y_{i,\bar q})$ for every $\bar q \in A_i \restriction \bar p$. Let $\dot z_i$ be the name consisting of pairs $(\bar r, \dot z)$ such that there are $\bar q \in A_i \restriction \bar p$, $\bar q' \in \mathbb{S}^{*\lambda}$, $ \bar r \leq \bar q, \bar q'$ and $(\bar q', \dot z) \in \dot y_{i,\bar q}$. Then $\bar p \Vdash \dot y_i = \dot z_i$ and it is easy to check that $H_\alpha \subseteq \sym(\dot z_i)$ and moreover that $\dot z_i \in \HS$. But then the standard name for the sequence $\langle \dot z_i^G : i < \xi \rangle$ is hereditarily symmetric as witnessed by $H_\alpha \in \sF$. When $\lambda \geq \omega_2$ is regular, we use the same argument but replace properness by the $\lambda$-cc of $\mathbb{S}^{*\lambda}$.

(2): Let $T \in M$ be an $\alpha$-closed tree without maximal elements, where $\alpha < \kappa$. By $(1)$, $T$ is also $\alpha$-closed in $V[G]$. Since $V[G]$ satisfies $\ZFC$ there is a branch $\langle a_\beta : \beta < \alpha \rangle \in V[G]$ through $T$. Again by $(1)$, $\langle a_\beta : \beta < \alpha \rangle \in M$.

(3): Let $\dot X \in \HS$ be a name for $X$, $H_\alpha \subseteq \sym(\dot X)$, $\gamma$ an ordinal and $\bar p \Vdash \gamma \in \dot X$. We claim that already $\bar p \restriction \alpha \Vdash \gamma \in \dot X$. Otherwise, let $\bar p' \leq \bar p \restriction \alpha$ be such that $\bar p' \Vdash \gamma \notin \dot X$. Consider a sequence $\bar \pi = \langle \dot \pi_\beta : \beta < \lambda \rangle$ defined recursively as follows. For $\beta < \alpha$, $\dot \pi_\beta$ is a name for the identity. For $\beta \geq \alpha$, we let $\dot \pi_\beta$ be a name for an automorphism $\pi_\beta$ such that $\pi_\beta \left( (\bar \pi \restriction \beta)(\dot p'(\beta))^K\right) \parallel \dot p(\beta)^K$, for any $\mathbb{S}^{*\beta}$-generic $K$. Then note that $\bar \pi (\bar p') \parallel \bar p$, $\bar \pi \in H_\alpha$ and by the symmetry lemma, $$\bar \pi (\bar p') \Vdash \gamma \notin \dot X.$$ This is obviously a contradiction. Thus $G \cap \mathbb{S}^{*\alpha} \in V[G \cap \mathbb{S}^{*\alpha}] = V[\langle x_\beta : \beta < \alpha \rangle] \subseteq M$ completely determines $X$.  

(4): Suppose that $\mathbb{R}$ is well-ordered in $M$. Then there is a set of ordinals that codes this well-order and every real appearing in it. But then this set exists in $V[\langle x_\beta : \beta < \alpha \rangle]$ for some $\alpha < \lambda$. This is impossible as new reals are added at later stages.

(5): For any $\alpha< \lambda$, let $\dot R_\alpha := \{ \dot x : \mathbbm{1} \Vdash \dot x \in \omega^\omega \wedge \dot x \text{ is an } \mathbb{S}^{*\alpha}\text{-name}\}$. Then note that $\sym(\dot R_\alpha) = \sG$. Clearly, $\dot R_\alpha^G = R_\alpha$. The canonical name for the sequence $\langle \dot R_\alpha^G : \alpha < \lambda \rangle$ is then easily seen to also be in $\HS$ as witnessed by $\sG \in \sF$. 

(6): Consider $\langle R_{\alpha +1} \setminus R_\alpha : \alpha < \lambda \rangle$. If $f \colon \lambda \to \omega^\omega$ is a choice function, then $f$ is coded by a set of ordinals. Thus $f \in V[\langle x_\beta : \beta < \alpha \rangle]$ for some $\alpha < \lambda$. This is clearly impossible, as $f(\alpha) \notin V[\langle x_\beta : \beta < \alpha \rangle]$ must be the case.
\end{proof}

$\omega_2$ will be collapsed when $\lambda \geq \omega_3$, so the interesting case is usually $\lambda \leq \omega_2$. Let us also note that (1) and (2) follow from a more general result, that can be abstracted from the proof. For example when $\mathbb{P}$ is proper and $\sF$ is countably closed, then $M$ is closed under countable sequences. And when $\mathbb{P}$ is $\lambda$-cc and $\sF$ is $<\lambda$-closed then $M$ is closed under $<\lambda$-sequences (see \cite[Lemma 3.4]{Karagila2019a}).

\subsection{Universally Baire sets}

In the following, $\mathbb{C}$ denotes \emph{Cohen forcing}, i.e. $\omega^{<\omega}$ ordered by extension.\footnote{Depending on the context, which should always be clear, $\mathbb{C}$ could also denote $2^{<\omega}$ or another non-trivial countable poset.} Whenever $\alpha$ is an ordinal and $T$ is a tree on $\omega\times\alpha$, $p[T]$ denotes the projection of $[T]$ on the first coordinate, i.e. $$p[T] := \{ x \in \omega^\omega : \exists w \in \alpha^\omega \forall n \in \omega ( (x\restriction n, w \restriction n) \in T) \}.$$

\begin{definition}[see e.g. \cite{Feng1992}]\label{def:omegauniv}
A set $A \subseteq \omega^\omega$ is called \emph{$\omega$-universally Baire} if there is some ordinal $\kappa$ and there are trees $T,U$ on $\omega \times \kappa$, so that 

\begin{enumerate}
    \item $p[T] = A$, $p[U] = \omega^\omega \setminus A$, 
    \item $\Vdash_{\mathbb{C}} p[\check T] \cup p[\check U] = \omega^\omega$.
\end{enumerate}
\end{definition}

The following is an easy absoluteness argument. 

\begin{remark}
If $A$, $T$ and $U$ are as above and $B = \omega^\omega \setminus A$, then $$\Vdash_{\mathbb{C}} \check{A} \subseteq p[\check T] \wedge \check{B} \subseteq p[\check U] \wedge p[\check T] \cap p[\check U] = \emptyset.$$ 
Thus in any Cohen extension of $V$, $p[T]$ and $p[U]$ form a partition of $\omega^\omega$.
\end{remark}

\begin{lemma}[\cite{Feng1992}]\label{lem:bairecontpre}
A set $A \subseteq \omega^\omega$ is $\omega$-universally Baire iff for every continuous $f \colon \omega^\omega \to \omega^\omega$, $f^{-1}(A)$ has the Baire property. 
\end{lemma}

Lemma~\ref{lem:bairecontpre} let's us define the pointclass of $\omega$-universally Baire sets on arbitrary Polish spaces $X$ as those sets $A \subseteq X$, such that for any $f \colon \omega^\omega \to X$ continuous, $f^{-1}(A)$ has the Baire property in $\omega^\omega$. In particular, this pointclass is closed under continuous preimages, countable unions and intersections and complements. On $\omega^\omega$ this coincides with Definition~\ref{def:omegauniv}. Assuming $\ZF + \DC$, all analytic and coanalytic sets are $\omega$-universally Baire, since they have tree representations and by absoluteness. Followingly, all $\sigma(\mathbf\Sigma^1_1 \cup \mathbf\Pi^1_1)$ sets (the $\sigma$-algebra generated by $\mathbf\Sigma^1_1 \cup \mathbf\Pi^1_1$) are $\omega$-universally Baire.

\begin{cor}
Assume that every projective set has the Baire property. Then every projective set is $\omega$-universally Baire. 
\end{cor}

The following lemma will let us apply \cite[Prop. 3.22]{Schilhan2020}.

\begin{lemma}\label{lem:bairelikescohen}
Let $M \preccurlyeq H(\theta)$ be countable, where $\theta$ is a ``large" regular cardinal, let $T,U,\kappa \in M$ be as in Definition~\ref{def:omegauniv} and let $G$ be generic over $M$ for a finite support product of Cohen forcing. Let $s \in \mathbb{C}$, $x \in M[G] \cap \omega^\omega$ and $\dot z \in M[G]$ be a $\mathbb{C}$-name for a real.\footnote{Formally, we should consider a forcing extension $N[G]$ of the transitive collapse $N$ of $M$ and translate the following statements according to the collapsing map. We omit this as is often usual in order to simplify the notation.} Then 
\begin{enumerate}
    \item $M[G] \models x \in p[T]$ iff $x \in p[T]$ and $M[G] \models x \in p[U]$ iff $x \in p[U]$,
    \item $M[G] \models `` s \Vdash \dot z \in p[\check T]"$ iff $s \Vdash \dot z \in p[\check T]$,
    \item for any continuous function $f \colon \omega^\omega \times \omega^\omega \to \omega^\omega$, the set $X = \{ y \in \omega^\omega : s \Vdash f(y,\dot z) \in p[\check T] \}$ is $\omega$-universally Baire.   
\end{enumerate}
\end{lemma}

\begin{proof}
(1): $x \in p[T]$ whenever $M[G] \models x \in p[T]$ by an easy upwards absoluteness argument. For the other direction, let us note that $M[G] \models p[T] \cup p[U] = \omega^\omega$. Namely, whenever $x \in M[G]$, there is a Cohen real $c \in M[G]$ over $M$ such that $x \in M[c]$. According to Definition~\ref{def:omegauniv} and as $M$ is elementary, $M[c] \models x \in p[T] \cup p[U]$ and in particular $M[G] \models x \in p[T] \cup p[U]$. Whenever $x \in p[T]$, we must have that $M[G] \models x \in p[T]$. Otherwise $M[G] \models x \in p[U]$ and therefore $x \in p[U]$ contradicting that $p[T] \cap p[U] = \emptyset$. The argument is the same for $p[U]$. 

(2): Assume that $M[G] \models `` s \Vdash \dot z \in p[\check T]"$. Then there is a name $\dot w \in M[G]$ consisting of pairs $(t,\check u)$, where $t \in \mathbb{C}$ and $u \in \kappa^{<\omega}$, such that $M[G] \models `` s \Vdash (\dot z,\dot w) \in [\check T]"$. By an easy absoluteness argument $s \Vdash (\dot z,\dot w) \in [\check T]$ and in particular, $s \Vdash \dot z \in p[\check T]$. On the other hand, if $M[G] \models `` s \not\Vdash \dot z \in p[\check T]"$, there is $t \leq s$ such that $M[G] \models `` t \Vdash \dot z \in p[\check U]"$. By the same argument as before, $t \Vdash \dot z \in p[\check U]$ and in particular $s \not\Vdash \dot z \in p[\check T]$.

(3): Let $\langle s_n : n \in \omega \rangle$ enumerate all conditions below $s$. Consider a tree $T'$ on $\omega \times \omega^{<\omega} \times\kappa^{<\omega} \times \mathbb{C}$ consisting of nodes $(u,v,w,t)$ such that if $\vert u \vert = \vert v \vert = \vert w \vert = \vert t\vert = n$, then $t(i) \leq_{\mathbb{C}} s_i$, $w(i) \in \kappa^i$, $(v(i),w(i)) \in T$ and $t(i) \Vdash f''([u\restriction i]\times [\dot z \restriction i]) \subseteq [v(i)]$ for every $i < n$, and whenever $t(i) \subseteq t(j)$ for $i,j < n$, then $w(i) \subseteq w(j)$. It is easy to check that $y \in p[T']$ (the projection to the first coordinate) iff $s \Vdash f(y,\dot z) \in p[\check T]$. Moreover, this holds in any forcing extension of $V$. Similarly, construct for every $t \leq s$ a tree $U'_t$ such that $y \in p[U'_t]$ iff $t \Vdash f(x,\dot z) \in p[\check U] $. This easily gives rise to a tree $U'$ such that in any forcing extension of $V$, $x \in p[U']$ iff $\exists t \leq s( t \Vdash (x,\dot z) \in p[\check U])$. Since $\Vdash_{\mathbb{C}} p[\check U] = \omega^\omega\setminus p[\check T]$ we have that $x \in p[U']$ iff $s \not\Vdash f(x,\dot z) \in p[\check T]$. Also, the same holds in any Cohen forcing extension since forcing with $\mathbb{C}$ twice is the same as forcing with it once. Followingly, $X = p[T']$, $\omega^\omega \setminus X = p[U']$ and in any Cohen extension $p[T'] = \omega^\omega \setminus p[U']$. Thus $X$ is $\omega$-universally Baire as witnessed by $T'$ and $U'$.
\end{proof}

\subsection{Previous results}
In this section, we will review some of the results from \cite{Schilhan2020}, that we want to apply. The main ingredient in the proof of Theorem~\ref{thm:mainold} is the following remarkable property of the Sacks iteration.

\begin{prop}[see {\cite[Proposition~4.23]{Schilhan2020}}]
\label{prop:propheart}
    Let $\lambda$ be an ordinal. Let $E$ be an analytic hypergraph on $\omega^\omega$, $\dot y$ a $\mathbb{S}^{*\lambda}$-name and $\bar p \in \mathbb{S}^{*\lambda}$, $\bar p \Vdash \dot y \in \omega^\omega$. Then there is $\bar q \leq \bar p$ and a compact $E$-independent set $K \subseteq \omega^\omega$ such that 
	
	\begin{enumerate}
		\item either $\bar q \Vdash \dot y \in K$,
		\item or $\bar q \Vdash \{ \dot y \} \cup K \text{ is not } E\text{-independent}$.
	\end{enumerate}
\end{prop}

In Section~\ref{sec:genabs}, we will prove analogues of this for projective and $L(\mathbb{R})$ hypergraphs, that will be used in our main result. For this we need to first look at some key ingredients of Proposition~\ref{prop:propheart}. One of them is the notion of strong mutual Cohen genericity (mCg) over a given countable model of set theory $M$, which is a generalization of mutual genericity in Cohen forcing over $M$ relative to a sequence of Polish spaces such as $\langle 2^\omega : i < \alpha \rangle$. Essentially, $\bar x_0, \dots, \bar x_{n-1} \in (2^\omega)^\alpha$ are ($\langle 2^\omega : i < \alpha \rangle$)-mCg if they are concatenations of finitely many mutually added Cohen generics over $M$ in spaces of the form $(2^\omega)^{[\delta, \gamma)}$, for $\delta < \gamma < \alpha$. For example, when $\alpha = 2$ and $(c_0,c_1,c_2)$ is $(2^{<\omega})^3$-generic, then the sequences $(c_0,c_1)$, $(c_0,c_2)$ are mCg. The exact definition can be found in \cite[Section~3.3]{Schilhan2020}, but all that is relevant to us are the following two lemmas.

	\begin{lemma}
	\label{lem:mainlemmainf}
	Let $\alpha < \omega_1$ and $E$ be an $\omega$-universally Baire hypergraph on $(2^\omega)^\alpha$. Then there is a countable model $M$, $\alpha +1 \subseteq M$, so that either
	
	\begin{enumerate}
		\item for any $\bar x_0,\dots, \bar x_{n-1} \in (2^\omega)^\alpha$ that are strongly $\langle 2^\omega : i < \alpha \rangle$-mCg over $M$, $$\{\bar x_0,\dots, \bar x_{n-1}\} \text{ is } E \text{-independent}$$
	\end{enumerate}
	
	or for some $N\in \omega$,
	\begin{enumerate}
		\setcounter{enumi}{1}
		\item there are $\phi_0, \dots, \phi_{N-1} \colon (2^\omega)^\alpha \to (2^\omega)^\alpha$ continuous, $\bar s \in \bigotimes_{i<\alpha} 2^{<\omega}$ so that for any $\bar x_0,\dots, \bar x_{n-1} \in (2^\omega)^\alpha \cap [\bar s]$ that are strongly mCg over $M$,  $$\{\phi_j(\bar x_i) : j < N, i<n\} \text{ is } E \text{-independent but } \{ \bar x_0\} \cup \{ \phi_j(\bar x_0) : j < N \} \in E.$$
		\end{enumerate}
	
\end{lemma}

Here, $\bigotimes_{i<\alpha} 2^{<\omega}$ consists of the finite sequences $\bar s$ of elements of $2^{<\omega}$ indexed by ordinals in $\alpha$. The sets $[\bar s] := \{ \bar x \in (2^\omega)^\alpha : \forall \beta \in \dom( \bar s) (s(\beta) \subseteq x(\beta)) \}$ form a basis for the product topology on $(2^\omega)^\alpha$. $M$ is always thought to be an $\in$-model of some fragment of set theory that contains $\omega$, such as an elementary submodel of some $H(\theta)$. We can also just think of $M$ as being some countable set of dense subsets of finite support products of Cohen forcing, since this is all that Cohen genericity depends on.

\begin{proof}
This is Proposition~3.22 in \cite{Schilhan2020} in combination with Lemma~\ref{lem:bairelikescohen}.
\end{proof}

\begin{lemma}[see {\cite[Lemma 4.21]{Schilhan2020}}]\label{lem:gettingmcg}

Let $\lambda$ be an ordinal, $(X,C) \in \tilde{\mathbb{S}}^{*\lambda}$, $\bar s \in \bigotimes_{\beta < \tp(C)} 2^{<\omega}$ and $M$ a countable model containing $(X,C)$. Then there is $(Y,C) \leq (X,C)$ so that $\Phi_{X}''Y \subseteq [\bar s]$ and for any $\bar x_0, \dots, \bar x_{n-1} \in Y$, $\Phi_{X}(\bar x_0), \dots, \Phi_{X}(\bar x_{n-1}) $ are strongly mCg over $M$. 
\end{lemma}

Analogous results also hold for the finite products of Sacks forcing.

\section{Generic absoluteness}\label{sec:genabs}

\begin{definition}
Let $\Gamma$ be a pointclass. The \emph{comeager uniformization principle} for $\Gamma$, denoted $\CU(\Gamma)$,  states that for every set $X \subseteq \omega^\omega \times \omega^\omega$ in $\Gamma$ with non-empty sections there is a comeager set $B \subseteq \omega^\omega$ and a continuous function $f \colon B \to \omega^\omega$ such that $\forall x \in B ((x, f(x)) \in X)$.  
\end{definition}

Note that the $\CUP$ implies that every projective set of reals is Baire measurable. Namely, whenever $X \subseteq \omega^\omega$ is arbitrary, consider it's characteristic function $F \colon \omega^\omega \to 2$. Then there is a comeager set $B$ such that $f := F \restriction B$ is continuous on $B$. Moreover we can assume that $B$ is $G_\delta$ and thus $A := f^{-1}(1)$ is Borel and $A \triangle X \subseteq \omega^\omega \setminus B$ is meager. In his seminal paper \cite{Shelah84}, Shelah showed that $\CUP$ is consistent relative to $\ZFC$.

In the following, we say that $\mathbb{P}$-projective absoluteness holds, for some arbitrary forcing notion $\mathbb{P}$, if projective formulas are absolute between $V$ and $V^{\mathbb{P}}$.

\begin{lemma}\label{lem:absproj}
Assume that $\CUP$ holds. Then $\mathbb{P}$-projective absoluteness holds, where $\mathbb{P}$ is a countable support iteration of Sacks forcing. 
\end{lemma}

\begin{proof}
Let us consider $\mathbb{P} = \mathbb{S}^{*\lambda}$. By Shoenfield's absoluteness theorem $\mathbf{\Sigma}^1_2$ formulas are absolute between $V$ and $V^\mathbb{P}$. Assume we have shown that $\mathbf{\Sigma}^1_n$ formulas are absolute for some $n \geq 2$. Consider an arbitrary $\mathbf{\Sigma}^1_{n+2}$ formula '$\exists x \in \omega^\omega \forall y\in \omega^\omega \varphi(x,y)$' where $\varphi(x,y)$ is $\mathbf{\Sigma}^1_{n}$ and assume that $\bar p \Vdash \exists x\in \omega^\omega \forall y\in \omega^\omega \varphi(x,y)$. Then there is a $\mathbb{P}$-name $\dot x$ such that $\bar p \Vdash \forall y\in \omega^\omega \varphi(\dot x, y)$. By the continuous reading of names there is $(X,C) \in \tilde{\mathbb{S}}^{*\lambda}$ and a continuous function $f \colon X \to \omega^\omega$ such that $i_\lambda(X,C) \leq \bar p$ and $ i_\lambda(X,C) \Vdash \dot x = f(\bar{x}_{\operatorname{gen}} \restriction C)$. Suppose towards a contradiction that in $V$, $\forall x \in \omega^\omega\exists y\in \omega^\omega \neg\varphi(x,y)$. Let $\alpha := \tp(C)$. Then in particular, $\forall \bar x \in (2^\omega)^{\alpha} \exists y\in \omega^\omega \neg \varphi(f(\Phi_{X}^{-1}(\bar x)), y)$. By $\CUP$ there is a comeager $G_\delta$ set $B \subseteq (2^\omega)^\alpha$ and a continuous function $g \colon B \to \omega^\omega$ such that $\forall \bar x \in B \neg \varphi(f(\Phi_{\bar q}^{-1}(\bar x)), g(\bar x))$. If $M$ is a countable model containing $B$, then by Lemma~\ref{lem:gettingmcg} there is a condition $(Y,C) \leq (X,C)$ so that for any $\bar x \in Y$, $\Phi_{X}(\bar x)$ is Cohen generic over $M$ and thus is a member of $B$. Note that `$\forall \bar x \in B \neg \varphi(f(\Phi_{X}^{-1}(\bar x)), g(\bar x))$' is $\mathbf\Pi^1_n$ and thus absolute between $V$ and $V^\mathbb{P}$, by the inductive assumption. In particular, in $V^\mathbb{P}$, $\neg\varphi(f(\Phi_{X}^{-1}(\Phi_{X}(\bar x_{\operatorname{gen}} \restriction C))), g(\Phi_{X}(\bar x_{\operatorname{gen}} \restriction C)))$, whereby $\neg\varphi(f(\bar x_{\operatorname{gen}} \restriction C), g(\Phi_{X}(\bar x_{\operatorname{gen}} \restriction C)))$, and thus $\exists y\in \omega^\omega \neg \varphi(\dot x^G,y)$. This contradicts $i_\lambda(Y,C) \Vdash \forall y\in \omega^\omega \varphi(\dot x, y)$. 
\end{proof}

\begin{definition}
Let $V \subseteq W$ be transitive models of $\ZF$ with the same ordinals. Then we say that $L(\mathbb{R})$-absoluteness holds between $V$ and $W$ ($V \preccurlyeq_{L(\mathbb{R})} W$) if for every formula $\varphi$ with ordinals and reals from $V$ as parameters, $$L(\mathbb{R})^V \models \varphi \text{ iff } L(\mathbb{R})^W \models \varphi.$$ 
\end{definition}

Let us recall that every set in $L(\mathbb{R})$ is definable in $L(\mathbb{R})$ using a formula that has reals and ordinals as the only parameters.\footnote{Namely, $\HOD(\mathbb{R})^{L(\mathbb{R})} = L(\mathbb{R})$ as $L(\mathbb{R})$ is the smallest inner model containing all reals.}

\begin{definition}
Let $\kappa$ be an inaccessible cardinal. Then $\Coll(\omega, <\kappa) := \prod_{\delta < \kappa}^{<\omega} \delta^{<\omega} = \{ p : \dom p \in [\kappa]^{<\omega}, \forall \delta \in \dom p (p(\delta) \in \delta^{<\omega}) \}$ is the \emph{Lévy-collapse of $\kappa$}. Similarly we define $\Coll(\omega, < \gamma)$ and $\Coll(\omega, [\delta, \gamma))$ for any ordinals $\delta < \gamma$.
\end{definition}

Solovay showed that whenever $\kappa$ is inaccessible in $L$ and $H$ is $\Coll(\omega,<\kappa)$-generic over $L$, in $L(\mathbb{R})^{L[H]}$, $\ZF + \DC$ holds and every set of reals has the Baire property. We will show that $L(\mathbb{R})$-absoluteness holds between $L[H]$ and an extension by the Sacks iteration or product. Throughout the rest of this section we fix an inaccessible cardinal $\kappa$ in $L$.

\begin{definition}
Let $W$ be an inner model of $\ZF$. Then we say that $W$ is \emph{Solovay-like} if for every $\delta < \kappa$ and every real $r \in W$, there is $\gamma \in [\delta, \kappa)$ and a $\Coll(\omega, <\gamma)$-generic $K \in W$ over $L$ such that $r \in L[K]$.
\end{definition}

A similar definition appears in \cite{BagariaBosch2004}.

\begin{lemma}
$W$ is Solovay-like iff there is a $\Coll(\omega, <\kappa)$ generic $H$ over $L$ (possibly outside of $W$), such that $L(\mathbb{R})^W  = L(\mathbb{R})^{L[H]}$.
\end{lemma}

\begin{proof}
Work in an extension of $W$ in which $\kappa$ and $\mathbb{R} \cap W$ are countable. Enumerate all reals $\langle r_n : n \in \omega \rangle$ of $W$ and let $\langle \delta_n : n \in \omega \rangle$ be cofinal in $\kappa$. We recursively find $\gamma_n > \delta_n$ and a $\Coll(\omega, < \gamma_n)$ generic $H_n \in W$ over $L$ such that $r_n \in L[H_n]$ and $H_n \subseteq H_{n+1}$, for every $n \in \omega$. Given $H_n$, we have that $H_n, r_n \in L[r]$ for some real $r \in W$, as $H_n$ is coded over $L$ by a single real in $W$ ($\gamma_n$ is collapsed to $\omega$). As $W$ is Solovay-like we find $\gamma_{n+1} > \delta_{n+1}$, $\vert \gamma_{n+1} \vert^L > \gamma_n$ and a $\Coll(\omega, < \gamma_{n+1})$-generic $K$ over $L$ such that $r \in L[K]$. Since the quotient of $\Coll(\omega,<\gamma_{n+1})$ by any forcing of size smaller than $\gamma_{n+1}$ is isomorphic to $\Coll(\omega, <\gamma_{n+1}) \cong \Coll(\omega, [\gamma_n,\gamma_{n+1}))$ (see \cite{Jech2003}), we find a $\Coll(\omega, < \gamma_{n+1})$ generic $H_{n+1} \in W$ that extends $H_n$ and such that $L[H_{n+1}] = L[K]$. Note that whenever $A \subseteq \Coll(\omega,<\kappa)$ is a maximal antichain in $L$, then $A \subseteq \Coll(\omega, <\delta_n)$ for some $n \in \omega$, by the $\kappa$-cc of $\Coll(\omega, < \kappa)$. In particular, $H := \bigcup_{n \in \omega} H_n$ is $\Coll(\omega, < \kappa)$-generic over $L$. Any real is added in an initial segment of the forcing. Thus it is clear that $\mathbb{R} \cap W = \mathbb{R} \cap L[H]$ and in particular, $L(\mathbb{R})^W  = L(\mathbb{R})^{L[H]}$. The other direction of the lemma is clear.
\end{proof}

\begin{lemma}
Let $V \subseteq W$ be both Solovay-like. Then $V \preccurlyeq_{L(\mathbb{R})} W$.
\end{lemma}

\begin{proof}
Let $r \in V$ and $\varphi$ be a formula with $r$ and ordinals as parameters and assume that $L(\mathbb{R})^V \models \varphi$. By the previous lemma there is a generic $\Coll(\omega, <\kappa)$-generic $H$ over $L$ such that $L(\mathbb{R})^V = L(\mathbb{R})^{L[H]}$. By well-known homogeneity and factorization arguments for the Lévy-collapse, we have that $$L[r] \models \mathbbm{1}\Vdash_{\Coll(\omega, < \kappa)} ``L(\mathbb{R}) \models \varphi".$$
Again, by the previous lemma and standard factorization arguments, $L(\mathbb{R})^W = L(\mathbb{R})^{L[r][H']}$ for some $\Coll(\omega, < \kappa)$-generic $H'$ over $L[r]$. In particular, $L(\mathbb{R})^{W}\models \varphi.$\end{proof}

\begin{lemma}\label{lem:sacksissolovay}
Let $V$ be a $\Coll(\omega, <\kappa)$-generic extension of $L$ and let $W$ be an extension of $V$ by a countable support iteration of Sacks forcing. Then $W$ is Solovay-like. In particular, $V \preccurlyeq_{L(\mathbb{R})} W$.
\end{lemma}

\begin{proof}
Let $r \in W$ and $\delta < \kappa$ be arbitrary. Then there is $(X,C) \in V$ and $f \colon X \to \omega^\omega$ continuous such that $r = f(\bar x \restriction C)$, where $\bar x$ is the generic added over $V$. Since $f$, $(X,C)$ and $\Phi_X$ are coded by a real and $V$ is Solovay-like, there is $\gamma \geq \delta$ and a $\Coll(\omega, < \gamma)$-generic $H_\gamma \in V$ over $L$ such that $f, (X,C), \Phi_X \in L[H_\gamma]$. There are only countably many reals in $L[H_\gamma]$ from point of view of $V$. Thus using Lemma~\ref{lem:gettingmcg}, we can force that $\Phi_X^{-1}(\bar x \restriction C)$ is Cohen generic over $L[H_\gamma]$. Working in $W$, $r \in L[H_\gamma][c]$, where $c$ is a Cohen real over $L[H_\gamma]$. As $\Coll(\omega,<\gamma) \cong \Coll(\omega,<\gamma) \times \mathbb{C}$, this proves the lemma.
\end{proof}

When we use a variable for a projective or an $L(\mathbb{R})$ set in a forcing extension, as usual what we mean is that we reinterpret the defining formula of that set in that extension.

\begin{prop}
\label{prop:propheart2}
	Assume $\CUP$ holds and let $\mathbb{P}$ be a countable support iteration of Sacks forcing. Let $E$ be a projective hypergraph on $\omega^\omega$, $\dot y$ a $\mathbb{P}$-name and $\bar p \in \mathbb{P}$, $p \Vdash \dot y \in \omega^\omega$. Then there is $\bar q \leq \bar p$ and a compact $E$-independent set $K \subseteq \omega^\omega$ such that 
	
	\begin{enumerate}
		\item either $\bar q \Vdash \dot y \in K$,
		\item or $\bar q \Vdash \{ \dot y \} \cup K \text{ is not } E\text{-independent}$.
	\end{enumerate}
If $V$ is a $\Coll(\omega, <\kappa)$-generic extension of $L$, we can change ``projective" to ``$L(\mathbb{R})$".
\end{prop}

\begin{proof}[Proof of Proposition~\ref{prop:propheart2}]
\label{pf:pfofpropheart}
By Lemma~\ref{lem:crninfinite}, there is $(X,C) \in \tilde{\mathbb{S}}^{*\lambda}$ and a continuous function $f \colon X \to \omega^\omega$ such that $i_\lambda(X,C) \leq \bar p$ and $i_\lambda(X,C) \Vdash \dot y = f(\bar x_{\operatorname{gen}} \restriction C)$. Let $\alpha = \tp(C)$ and consider the hypergraph $E'$ on $(2^\omega)^\alpha$ such that $$e \in E' \leftrightarrow f''\Phi_{X}^{-1}(e) \in E.$$ $E'$ is a projective hypergraph on $(2^\omega)^\alpha$ and by Lemma~\ref{lem:mainlemmainf} there is a countable model $M$, $\alpha + 1 \subseteq M$, such that either 

	\begin{enumerate}
		\item for any $\bar x_0,\dots, \bar x_{n-1} \in (2^\omega)^\alpha$ that are strongly $\langle 2^\omega : i < \alpha \rangle$-mCg over $M$, $$\{\bar x_0,\dots, \bar x_{n-1}\} \text{ is } E' \text{-independent}$$
	\end{enumerate}
	
	or for some $N\in \omega$,
	\begin{enumerate}
		\setcounter{enumi}{1}
		\item there are $\phi_0, \dots, \phi_{N-1} \colon (2^\omega)^\alpha \to (2^\omega)^\alpha$ continuous, $\bar s \in \bigotimes_{i<\alpha} 2^{<\omega}$ so that for any $\bar x_0,\dots, \bar x_{n-1} \in (2^\omega)^\alpha \cap [\bar s]$ that are strongly mCg over $M$,  $$\{\phi_j(\bar x_i) : j < N, i<n\} \text{ is } E' \text{-independent but } \{ \bar x_0\} \cup \{ \phi_j(\bar x_0) : j < N \} \in E'.$$
		\end{enumerate}

Apply Lemma~\ref{lem:gettingmcg} to get $(Y,C) \leq (X,C)$ such that for any $\bar x_0, \dots, \bar x_{n-1} \in Y$, $$\Phi_X(\bar x_0), \dots, \bar \Phi_X(x_{n-1})$$ are strongly mCg over $M$. If (2) holds, we ensure that $\Phi_X''Y \subseteq [\bar s]$. 

Let $\bar q := i_\lambda(Y,C) \leq \bar p$. In case (1), let $K := f''Y$. Then $\bar q \Vdash \dot y \in K$. If $\bar x_0, \dots, \bar x_{n-1} \in Y$, then $\{ \Phi_X(\bar x_0), \dots, \Phi_X(\bar x_{n-1})\} \notin E'$ and thus $\{ f(\bar x_0), \dots, f(\bar x_{n-1}) \} \notin E$. Thus $K$ is $E$-independent. In case (2), let $K := \bigcup_{i < N} (\phi_i \circ \Phi_X)''Y$. By a similar argument, $K$ is $E$-independent and for every $\bar x \in Y$, $\{ f(\bar x)\} \cup K$ is not $E$-independent. By projective absoluteness, in $V^\mathbb{P}$, $\{ f(\bar x_{\operatorname{gen}}\restriction C)\} \cup K$ is not $E$-independent.

If $V$ is a $\Coll(\omega, <\kappa)$-generic extension of $L$, we can use $L(\mathbb{R})$-absoluteness and the fact the every set of reals in $L(\mathbb{R})$ is $\omega$-universally Baire to get the analogous result. 
\end{proof}

\section{\texorpdfstring{A $\Delta^1_4$ Bernstein set}{A Delta\textasciicircum 1\_2 Bernstein set}}
\label{sec:bernstein}

\begin{definition}
Let $x,y \in \omega^\omega$. Then we write $x \leq_L y$ iff $x \in L[y]$ and $x =_L y$ iff $x \leq_L y$ and $y \leq_L x$. We write $x <_L y$ iff $x \leq_L y$ and $y \not\leq_L x$. A set of the form $[x]_L := \{ y \in \omega^\omega : x =_L y\}$ is called an \emph{$L$-degree}.
\end{definition}

Note that the relations $\leq_L$ and $=_L$ are $\Sigma^1_2$ and thus absolute between inner models. 

\begin{lemma}[Folklore]\label{lem:degrees}
Let $G$ be $\mathbb{S}^{*\lambda}$-generic over $L$, where $\lambda \leq \omega_2$. Then in $L[G]$, $<_L$ is a well-order of type $\lambda$ or $\lambda +1$.
\end{lemma}

\begin{proof}
See for example \cite[Corollary 14]{Groszek1988} or \cite{Kanovei1999}.
\end{proof}

\begin{thm}\label{thm:bernstein}
Let $M$ be a symmetric extension of $L$ as constructed in Section~\ref{sec:symit} with $\lambda \leq \omega_1$. Then there is a $\Delta^1_4$ Bernstein set in $M$. 
\end{thm}

\begin{proof}
By Lemma~\ref{lem:symit}, there is no well-ordering of the continuum in $M$. Thus there is no maximal $L$-degree in $M$ and the order type of $<_L$ in $M$ is a limit ordinal $\kappa \leq \omega_1$.\footnote{In fact, $\kappa = \lambda$.} Consider the sequence $\langle d_\alpha : \alpha < \kappa \rangle$ of $L$-degrees in $M$ ordered by $<_L$. We claim that $B := \bigcup_{\alpha < \kappa} d_{2\alpha}$ is a Bernstein set. Namely, if $T \subseteq \omega^{<\omega}$ is a perfect tree, then $[T]_L = d_\beta$ for some $\beta < \kappa$. But then new elements of $[T]$ exist in any degree above $d_\beta$.\footnote{Whenever $X$ is any perfect set, new elements are added to it in any extension with new reals.} In particular, we are missing out all the elements of $[T]$ in odd degrees but including all those in even degrees. Thus $[T]$ is included neither in $B$ nor its complement.  

Let us show that $B$ is $\Delta^1_4$. In order to express that $y$ is in the $\beta$'th degree it suffices to say that there is a sequence $\langle y_\xi : \xi \leq \beta \rangle$ of reals such that 

\begin{enumerate}
    \item $0 =_L y_0$, 
    \item for every $\xi \leq \beta$, for every $z \leq_L y_\xi$, $\exists \xi' \leq \xi (z =_L y_\xi)$, 
    \item and $y =_L y_\beta$. 
\end{enumerate}

Thus $y \in B$ iff there is an odd countable ordinal $\beta$, and a sequence $\langle y_\xi : \xi \leq \beta \rangle$ such that (1), (2) and (3) are satisfied. Using standard coding techniques (e.g. to talk about countable ordinals) we see that this is expressed by a $\Sigma^1_4$ formula. 
On the other hand, $y \in B$ iff for any countable ordinal $\beta$ and any sequence $\langle y_\xi : \xi \leq \beta \rangle$ such that (1), (2) and (3), $\beta$ is odd.
This is expressed by a $\Pi^1_4$ formula. Thus $B$ is $\Delta^1_4$.
\end{proof}

\section{Main results}

\begin{thm}
\label{thm:mainit}
Let $V \models \CUP \wedge \GCH$, $\lambda$ be an ordinal and let $M$ be the symmetric extension of $V$ from Section~\ref{sec:symit}. Then the following hold in $M$: 

\begin{enumerate}
    \item $\DC_{\omega_1} + \neg \WO(\mathbb{R})$, if $\lambda = \omega_2$.
    \item $\neg \AC_\omega(\mathbb{R})$, if $\lambda = \omega$.
    \item Every projective hypergraph on the reals has a maximal independent set that is a union of $\aleph_1$-many compact sets. 
    \item In particular, every projective set is the union of $\aleph_1$-many compact sets. 
    \item There is a Bernstein set.
    \item There is a tower, a scale, a Luzin set, a Sierpiński set and a P-point.
    \item There is a two-point set.
\end{enumerate}

If $V$ is a $\Coll(\omega, <\kappa)$-generic extension of $L$, where $\kappa$ is inaccessible in $L$, we can strengthen ``projective" to $``L(\mathbb{R})"$ and every equivalence relation in $L(\mathbb{R})$ has a transversal in $M$.
\end{thm}

\begin{proof}
(1) and (2) is Lemma~\ref{lem:symit}.

(3) First let us note that for every $n \in \omega$ there is a universal lightface $\Sigma^1_n$ hypergraph $E$ on $\omega^\omega\times \omega^\omega$ coded in $V$. More precisely, for every $r \in \omega^\omega$, there is a hypergraph $E_r$ on $\{ r \} \times \omega^\omega$ such that $E = \bigcup_{r \in \omega^\omega} E_r$, and for every $\Sigma^1_n$ hypergraph $E'$ on $\omega^\omega$ there is some $r$ such that $e \in E'$ iff $\{r\} \times e \in E_r$. It is clear that it is sufficient and necessary to produce a maximal independent set for $E$. Moreover, for simplicity we can assume that $E$ lives on $\omega^\omega$.  

Working in $V$, let $\langle X_\alpha, C_\alpha, f_\alpha : \alpha < \omega_1 \rangle$ enumerate all triples $(X, C, f)$, such that $(X,C) \in \tilde{\mathbb{S}}^{*\omega_1}$ and $f \colon X \to \omega^\omega$ is continuous. We will construct a sequence of compact sets $\langle K_\alpha : \alpha < \omega_1 \rangle \in V$ recursively such that for every $\alpha < \omega_1$, $\bigcup_{\beta < \alpha} K_\beta$ is $E$-independent. Suppose that $\langle K_\beta : \beta < \alpha \rangle$ is given. Then consider the hypergraph $E_\alpha$ given by $$e \in E_\alpha \leftrightarrow e \cup \bigcup_{\beta < \alpha} K_\beta \text{ is } E\text{-independent}.$$

Then $E_\alpha$ is projective and we can apply Proposition~\ref{prop:propheart2} to $\bar p_\alpha := i(X_\alpha,C_\alpha)$ and a name $\dot y$ for $f_\alpha(\bar x_{\operatorname{gen}} \restriction C)$, where $i$ is the dense embedding $i_{\omega_1}$ from Lemma~\ref{lem:crninfinite}. Thus there is $\bar q_\alpha \leq \bar p_\alpha$ and a compact $E_\alpha$-independent set $K$ such that either 

$$\bar q_\alpha \Vdash \dot y \in K,$$
or $$\bar q_\alpha \Vdash \{ \dot y \} \cup K \text{ is not } E\text{-independent}.$$

We let $K_\alpha := K$. Since $K_\alpha$ is $E_\alpha$-independent, $\bigcup_{\beta \leq \alpha} K_\alpha$ is indeed $E$-independent. Moreover this holds true in $V^{\mathbb{P}}$ by Lemma~\ref{lem:absproj}, where $\mathbb{P} = \mathbb{S}^{*\lambda}$. Let us check that in $V^\mathbb{P}$, $X := \bigcup_{\alpha < \omega_1} K_\alpha$ is maximal. To this end, let $\dot y$ be a name for a real, $\bar p \in \mathbb{P}$ and assume that $$\bar p \Vdash \dot y \notin X \wedge \{ \dot y \} \cup X \text{ is } E \text{-independent}.$$ 

It is sufficient to assume that $\bar p \in \mathbb{S}^{*\omega_1}$ and that $\dot y$ is an $\mathbb{S}^{*\omega_1}$-name.\footnote{Specifically, note in the proof of Proposition~\ref{prop:propheart2}, that the construction of $K$ does not depend at all on the domain of $X$ and $C$.} By Lemma~\ref{lem:crninfinite}, there is $\alpha < \omega_1$, such that $\bar p_\alpha \leq \bar p$ and $\bar p_\alpha \Vdash \dot y = f_\alpha(\bar x_{\operatorname{gen}} \restriction C_\alpha)$. But then $\bar q_\alpha \leq \bar p_\alpha \leq \bar p$ and $\bar q_\alpha \Vdash \dot y \in K_\alpha \subseteq X$ or $\bar q_\alpha\Vdash ``\{\dot y \} \cup X$ is not $E$-independent". Both options yield a contradiction. Finally, note that projective absoluteness holds between $M$ and $V^{\mathbb{P}}$. Namely, $\mathbb{R} \cap M = \bigcup_{\beta < \lambda} \mathbb{R} \cap V^{\mathbb{S}^{*\beta}}$ and projective absoluteness holds between $V^{\mathbb{S}^{*\beta}}$ and $V^{\mathbb{P}}$, for every $\beta < \lambda$. Thus $\bigcup_{\alpha < \omega_1} K_\alpha$ is maximal $E$-independent in $M$.

(4): For a projective set $X$, consider the hypergraph $E := \{ \{x \} : x \notin X \}$.

(5): Consider the sequence $\langle R_\alpha : \alpha < \lambda\rangle$ from Lemma~\ref{lem:symit}, where $R_\alpha$ is the set of reals added up to the $\alpha$'th stage of the iteration. We claim that $B := \bigcup_{\alpha < \lambda} R_{2\alpha +1} \setminus R_{2\alpha}$ is a Bernstein set. Namely, if $T \subseteq \omega^{<\omega}$ is a perfect tree, then $T$ is added at some stage $\beta$. But then new elements of $[T]$ are added at each stage after $\beta$. In particular, we are missing out all the elements added later at even stages and including all those added at odd stages. Thus $[T]$ is included neither in $B$ nor its complement.  

(6): This follows from well-known preservation results for Sacks forcing. See e.g. \cite[Corollary 3.4]{brendle2018model} for Luzin and Sierpiński sets. 

(7): This follows from \cite{BeriashviliSchindler} and Lemma~\ref{lem:symit}.

Now assume that $V$ is a $\Coll(\omega, <\kappa)$-generic extension of $L$. Then note that $M$ is Solovay-like since every real in $M$ is contained in a Solovay-like $W \subseteq M$, by Lemma~\ref{lem:sacksissolovay}. Thus we have appropriate absoluteness and we can use a similar construction as in (3). This time, for every formula $\varphi(x,y,\bar \alpha)$, we produce a set $X\subseteq (\omega^\omega)^2$ such that a vertical section $X_r$, for $r \in \omega^\omega$, is maximal independent for the hypergraph defined by the formula $\varphi(x,r,\bar \alpha)$.

Whenever $E \in L(\mathbb{R})$ is an equivalence relation on a set $X \in L(\mathbb{R})$, there is $\alpha$ such that $E$, $X$ and all elements of $X$ are definable in $W_\alpha := (V_\alpha)^{L(\mathbb{R})}$ from reals and ordinals. For any formula $\varphi$ and a finite sequence of ordinals $\bar \beta$ below $\alpha$, we find a transversal $X_{\varphi, \bar \beta}$ for the equivalence relation $$E_{\varphi, \bar \beta} := \{ (x,y) \in (\omega^\omega)^2 : x = y \vee \left(\varphi(W_\alpha,x,\bar \beta), \varphi(W_\alpha,y,\bar \beta)\right) \in E \},$$ where $\varphi(W_\alpha,w,\bar \beta) := \{z \in W_\alpha : W_\alpha \models \varphi(z,w,\bar \beta) \}$.\footnote{It might happen that $\varphi(W_\alpha,x,\bar \beta)$ or $\varphi(W_\alpha,y,\bar \beta)$ are not elements of $X$ which is why we add the clause $x=y$ to ensure that $E_{\varphi, \bar \beta}$ is reflexive.} Moreover we can assume that the map that sends a pair $(\varphi, \bar \beta)$ to $X_{\varphi, \bar \beta}$ is in $M$, since the definitions of the sets $X_{\varphi, \bar \beta}$ can be chosen uniformly in $V$. Let $\langle (\varphi_{\xi}, \bar \beta_\xi) : \xi < \mu \rangle$ be an enumeration in $M$ of all pairs $(\varphi, \bar \beta)$. Then we let \begin{multline*}
    Y := \{ y \in X : \exists \xi < \mu \exists x \in X_{\varphi_\xi, \bar \beta_\xi}( y = \varphi(W_\alpha,x, \bar \beta) ) \\ \wedge \forall \xi' < \xi \forall x \in X_{\varphi_{\xi'}, \bar \beta_{\xi'}} \left( (\varphi(W_\alpha,x,\bar \beta), \varphi(W_\alpha,y,\bar \beta)) \notin E \right) \}.  
\end{multline*}

It is clear that $Y$ is an $E$-transversal. 
\end{proof}

\begin{remark}
In general, when $\AC_\omega(\mathbb{R})$ fails, ``projective" and $\mathbf{\Sigma}^1_{\omega}$ need not be equivalent. Nevertheless, in the models of Theorem~\ref{thm:mainit} this holds. Namely, when $X \subseteq \omega^\omega \times \omega^\omega$ is a universal $\Sigma^1_1$ set, the equivalence relation $x \sim y$ iff $X_x = X_y$ is $\Pi^1_2$ and has a transversal. Thus whenever $\langle B_n : n \in \omega \rangle$ is a sequence of $\Delta^1_1$ sets, we can find $\langle x_n : n \in \omega \rangle$ such that $B_n = X_{x_n}$ for every $n \in \omega$. Thus $\bigcup_{n \in \omega} B_n$ is $\mathbf{\Sigma^1_1}$ and similarly it is $\mathbf{\Pi}^1_1$, so $\mathbf{\Delta}^1_1$. This shows that the Borel sets (the smallest $\sigma$-algebra containing the open sets) are $\mathbf{\Delta}^1_1$. In particular, projective sets are $\mathbf{\Sigma}^1_\omega$. 
\end{remark}

\begin{thm}
Let $\lambda \leq \omega_1$ and $M$ be the symmetric extension of $V=L$ from Section~\ref{sec:symit}. Then the following hold in $M$: 
\begin{enumerate}
    \item Every $\Sigma^1_1(r)$ hypergraph has a $\Delta^1_2(r)$ maximal independent set.
    \item In particular, there is a $\Delta^1_2$ Hamel basis, transcendence basis, Vitali set, etc. 
    \item There is a $\Pi^1_1$ tower and scale.
    \item There is a $\Delta^1_2$ Luzin set, Sierpiński set and P-point.
    \item There is a $\Delta^1_4$ Bernstein set.
    \end{enumerate}
\end{thm}

\begin{proof}
(1): When $\lambda = \omega_1$, then $M$ has the same reals as $V^\mathbb{P}$ and this follows directly from \cite[Theorem 5.1]{Schilhan2020}. If $\cf(\lambda) = \omega$ and $E$ is $\Sigma^1_1(r)$, $r \in M$, then there is a $\Delta^1_2(r)$ maximal $E$-independent set in $V^\mathbb{P}$ according to \cite[Theorem 5.1]{Schilhan2020}. By Shoenfield-absoluteness, the same definition yields a maximal independent set in $M$. (3)-(4): There are such sets in $L$ and their properties are preserved. (5): This is Theorem~\ref{thm:bernstein}.
\end{proof} 

\begin{remark}
Assuming suitable large cardinals exist, one can also prove level by level versions of our results. For example, when $x^\sharp$ exists for every real $x$, the least inner model $W$ that is closed under sharps has a $\Sigma^1_3$-good well-order of the reals of length $\omega_1$ (see e.g. \cite[Fact 2.109]{BagariaWoodin1997}). Also, the existence of sharps implies that $\mathbf\Sigma^1_2$ sets have the Baire property in $W$ and by Kondo-uniformization $\CU(\mathbf{\Sigma}^1_2)$ follows. Modifying the results where appropriate we consistently obtain that every $\mathbf{\Sigma}^1_2$ hypergraph has a $\mathbf\Delta^1_3$ maximal independent set, there is a $\Delta^1_5$ Bernstein set and there is no well-ordering of the continuum. 
\end{remark}

\section{Open problems}

\begin{quest}
In Theorem~\ref{thm:main1}, can we replace (6) by ``every hypergraph in $L(\mathbb{R})$ has a maximal set"? Do we need the inaccessible? 
\end{quest}

\begin{quest}
In Theorem~\ref{thm:main2}, can we lower the complexity of a Bernstein set to $\Delta^1_3$?
\end{quest}

How much further can we push the existence of maximal sets without inducing well-orders? Consider the following principle. $\HYP(\mathbb{R})$ is saying that every hypergraph on $\mathbb{R}$ has a maximal independent set. 

\begin{quest}
Is $\HYP(\mathbb{R}) + \neg \WO(\mathbb{R})$ consistent? What if we add $\DC$?
\end{quest}

It is unlikely that our approach will work here. We highly rely on that the hypergraphs in consideration have the Baire property, but at the same time maximal sets are typically highly unregular (e.g. a Vitali set) and it is unclear how to then deal with hypergraphs constructed from such sets.

A weakening of $\HYP(\mathbb{R})$ is the \emph{continuum splitting property} that we denote by $\SP(\mathbb{R})$.\footnote{We also like to call it ``selector principle".} It is saying that every equivalence relation on $\mathbb{R}$ has a transversal. Note that $\SP(\mathbb{R})$ implies $\AC_\omega(\mathbb{R})$. In particular, $\aleph_1$ must be regular. Also, it implies that there is a set of reals of size $\aleph_1$. To our knowledge it is also still unknown whether $\SP(\mathbb{R}) + \DC + \neg \WO(\mathbb{R})$ is consistent.

\providecommand{\bysame}{\leavevmode\hbox to3em{\hrulefill}\thinspace}
\providecommand{\MR}{\relax\ifhmode\unskip\space\fi MR }
\providecommand{\MRhref}[2]{%
  \href{http://www.ams.org/mathscinet-getitem?mr=#1}{#2}
}
\providecommand{\href}[2]{#2}

\end{document}